\documentclass[reqno,11pt]{amsart}
\usepackage{amsmath, latexsym, amsfonts, amssymb, amsthm, amscd,color}
\usepackage{graphics,epsf,psfrag}
\numberwithin{equation}{section}

\newtheorem{theorem}{Theorem}[section]
\newtheorem{lemma}[theorem]{Lemma}
\newtheorem{proposition}[theorem]{Proposition}

\newtheorem{rem}[theorem]{Remark}

\newcommand{\ind}{\mathbf{1}}

\newcommand{\R}{\mathbb{R}}

\newcommand{\N}{\mathbb{N}}
\renewcommand{\tilde}{\widetilde}

\newcommand{\cL}{{\ensuremath{\mathcal L}} }

\newcommand{\bP}{{\ensuremath{\mathbf P}} }
\newcommand{\bE}{{\ensuremath{\mathbf E}} }

\DeclareMathSymbol{\leqslant}{\mathalpha}{AMSa}{"36} % nicer `smaller or equal'
\DeclareMathSymbol{\geqslant}{\mathalpha}{AMSa}{"3E} % nicer `larger or equal'
\DeclareMathSymbol{\eset}{\mathalpha}{AMSb}{"3F}     % nicer `emptyset'
\newcommand{\dd}{\,\text{\rm d}}             % a straight d for differentials
\newcommand{\bbE}{{\ensuremath{\mathbb E}} }

\newcommand{\bbL}{{\ensuremath{\mathbb L}} }

\newcommand{\bbP}{{\ensuremath{\mathbb P}} }

\newcommand{\ga}{\alpha}
\newcommand{\gb}{\beta}
\newcommand{\gga}{\gamma}            % \gg already exists...
\newcommand{\gd}{\delta}
\newcommand{\gep}{\varepsilon}       % \ge already exists...

\newcommand{\gD}{\Delta}

\newcommand{\go}{\omega}

\makeatletter
\def\captionfont@{\footnotesize}
\def\captionheadfont@{\scshape}

\long\def\@makecaption#1#2{%
  \vspace{2mm}
  \setbox\@tempboxa\vbox{\color@setgroup
    \advance\hsize-6pc\noindent
    \captionfont@\captionheadfont@#1\@xp\@ifnotempty\@xp
        {\@cdr#2\@nil}{.\captionfont@\upshape\enspace#2}%
    \unskip\kern-6pc\par
    \global\setbox\@ne\lastbox\color@endgroup}%
  \ifhbox\@ne % the normal case
    \setbox\@ne\hbox{\unhbox\@ne\unskip\unskip\unpenalty\unkern}%
  \fi
  \ifdim\wd\@tempboxa=\z@ % this means caption will fit on one line
    \setbox\@ne\hbox to\columnwidth{\hss\kern-6pc\box\@ne\hss}%
  \else % tempboxa contained more than one line
    \setbox\@ne\vbox{\unvbox\@tempboxa\parskip\z@skip
        \noindent\unhbox\@ne\advance\hsize-6pc\par}%
\fi
  \ifnum\@tempcnta<64 % if the float IS a figure...
    \addvspace\abovecaptionskip
    \moveright 3pc\box\@ne
  \else % if the float IS NOT a figure...
    \moveright 3pc\box\@ne
    \nobreak
    \vskip\belowcaptionskip
  \fi
\relax
}
\makeatother
\def\writefig#1 #2 #3 {\rlap{\kern #1 truecm
\raise #2 truecm \hbox{#3}}}

\newcommand{\tf}{\textsc{f}}

\newcommand{\rev}{}
\begin{document}

\title[Hierarchical pinning model with site disorder]{Hierarchical pinning model with site disorder: Disorder is marginally relevant}
\author{Hubert Lacoin} 
\address{
  Universit{\'e} Paris Diderot and Laboratoire de Probabilit{\'e}s et Mod\`eles Al\'eatoires (CNRS U.M.R. 7599),
U.F.R.                Math\'ematiques, Case 7012 (Site Chevaleret),
                75205 Paris cedex 13, France
}

\email{lacoin\@@math.jussieu.fr}

\begin{abstract}
We study a hierarchical disordered pinning model with site disorder for which, like in the bond disordered case \cite{cf:DHV,cf:GLT}, there exists a  value of a 
parameter $b$ (which enters in the definition of the hierarchical lattice) that
separates an {\sl irrelevant disorder} regime and a {\sl relevant disorder} regime.
We show that for such a value of $b$ the critical point of the disordered system
 is different from the critical point of the annealed version of the model. The proof
goes beyond the technique used in  \cite{cf:GLT} and it takes explicitly 
advantage of the inhomogeneous character of the Green function of the model. 
\end{abstract}



\maketitle
 \section{The model and the results}

\subsection{A quick survey and some motivations}

A lot of progress have been made recently in the understanding of pinning models (see \cite{cf:Book} for a  survey and particularly Chapter $1$ for a definition of the simplest random walk based model) in particular in comparing quenched and annealed critical point (see \cite{cf:Ken, cf:KZ,cf:DGLT, cf:T_cmp,  cf:T_fractmom}).
\rev{But these works do not settle the question of whether the quenched and annealed critical points
coincide or not for the simple random walk based model}. Moreover, physicists predictions on such an issue do not agree (see for example \cite{cf:DHV, cf:FLNO}). The reason why the question is not settled 
 for the random walk based model lies in the exponent ($3/2$) of the 
 law of the first return to zero: for smaller (respectively larger) values of the exponent
 a heuristic argument ({\sl Harris criterion}, \cite{cf:DHV, cf:FLNO}), 
 made rigorous in \cite{cf:Ken, cf:KZ, cf:DGLT, cf:T_cmp}, tells 
 us that annealed and quenched critical points coincide at high
 temperature (respectively, they differ at all temperatures). The first scenario
 goes under the name of irrelevant disorder regime and the second as 
 relevant disorder regime. The irrelevant disorder regime is also characterized
 by the fact that quenched and annealed critical exponents coincide
 \cite{cf:Ken, cf:T_cmp}, while for the relevant disorder they are different
 \cite{cf:GT_cmp}. The terminology {\sl relevant/irrelevant} comes from
 the renormalization group arguments \cite{cf:DHV, cf:FLNO} leading to the Harris criterion
 and in such a context the {\sl undecided} case is called {\sl marginal}
 and, in general, it poses a very challenging problem even at a heuristic level 
 (see {\sl e.g.} \cite{cf:DG} and references therein).  

Much work has been done on the statistical mechanics on 
a particular class of hierarchical lattices, the {\sl diamond lattices}, 
because of the explicit form of the renormalization group
transformations on such lattices, while often retaining a clear link
with the corresponding non hierarchical lattices \cite{cf:DG}.
A hierarchical model for disordered pinning has been explicitly 
considered in  \cite{cf:DHV} and a rigorous analysis of this model 
has been taken up in  \cite{cf:GLT}, but such a rigorous analysis cannot confirm the
prediction in  \cite{cf:DHV}   that the disorder is relevant also in the marginal 
regime (more precisely, in \cite{cf:DHV} it is claimed that annealed and quenched
critical points differ at marginality).
It should however be pointed out that the model in \cite{cf:DHV,cf:GLT}
is a {\sl bond} disorder model and there is a natural companion to such a model,
that is the one in which the disorder is on the sites.  A priori there is no
particular reason to choose either of the two cases, but, if we take a closer look, the site disorder case is somewhat closer to the non hierarchical
case. The reason is that the Green function (see below) of the
bond model is constant through the lattice, while the Green function of the site model is not
 (this analogy can be pushed further, see Remark \ref{rem:Green}). 
In this paper we analyze the hierarchical pinning model with site disorder  
  and  we  establish disorder relevance in the marginal regime. \rev{It seems unlikely that our method can be adapted in a straightforward way to settle the question
for the bond model or for the random walk based model, but one can use it to improve
the result in \cite{cf:DGLT} (we will come back to this point in Remark \ref{remrev} below). This result is a confirmation (although the setup we consider here is slightly different) to the claim made in \cite{cf:DHV} that disorder is relevant at any temperature when the specific heat exponent vanishes.}

\subsection{The model}

Let $(D_n)_{n\in\N}$ be the sequence of lattices defined as follow
\begin{itemize}
	\item $D_0$ is made of one single edge linking two points $A$ and $B$.
	\item $D_{n+1}$ is obtained for $D_n$ by replacing each edge by $b$ branches of $s$ edges (with $b$ and $s$ in $\left\{2,3,4,\dots,\right\}$).
\end{itemize}
On $D_n$ we fix one directed path $\sigma $ linking $A$ and $B$ ({\sl the wall}).

Given $\gb>0$, $h\in \R$ and $\{\go_i\}_{i\in\N}$ a sequence of i.i.d.\ random variables (\rev{with law $\bbP$ and expectation denoted by $\bbE$}) with zero mean, unit variance, and satisfying
\begin{equation}
M(\gb):=\bbE\left[\exp(\gb\go_1)\right]<\infty\quad \text{ for every } \gb>0,
\end{equation} 
one defines the partition function of the system of rank $n$ by
\begin{equation}
\label{eq:part}
R_n\, =\, R_n(\gb,h)\, :=\, \bE_n\left[\exp(H_{n,\go,\gb,h}(S)\right], 
\end{equation}
 where  $\bP_n$ is the uniform probability on all directed path $S=(S_i)_{0 \le i\le s^n}$ on $D_n$ linking $A$ to $B$, \rev{and $\bE_n$ the related expectation}, and 
 \begin{equation}
 H_{n,\go,\gb,h}(S)\, =\, \sum_{i=1}^{s^n-1} \left[\gb\go_i+h-\log M(\gb)\right]\ind_{\{S_i=\sigma_i\}} .
\end{equation}

\begin{figure}[h]
\begin{center}
\leavevmode
\epsfysize =6.5 cm
\psfragscanon
\psfrag{o1}[c]{{\small$\go_1$}}
\psfrag{o2}[c]{{\small$\go_2$}}
\psfrag{o3}[c]{{\small$\go_3$}}
%\psfrag{wall}[c]{{\tiny The path $\sigma$ (the {\sl wall})}}
\psfrag{a}[c]{\normalsize A}
\psfrag{b}[c]{B}
\psfrag{l0}[c]{$D_0$}
\psfrag{l1}[c]{$D_1$}
\psfrag{l2}[c]{$D_2$}
\psfrag{traject}[l]{$S$ {\tiny (a directed path on $D_2$)}}

\epsfbox{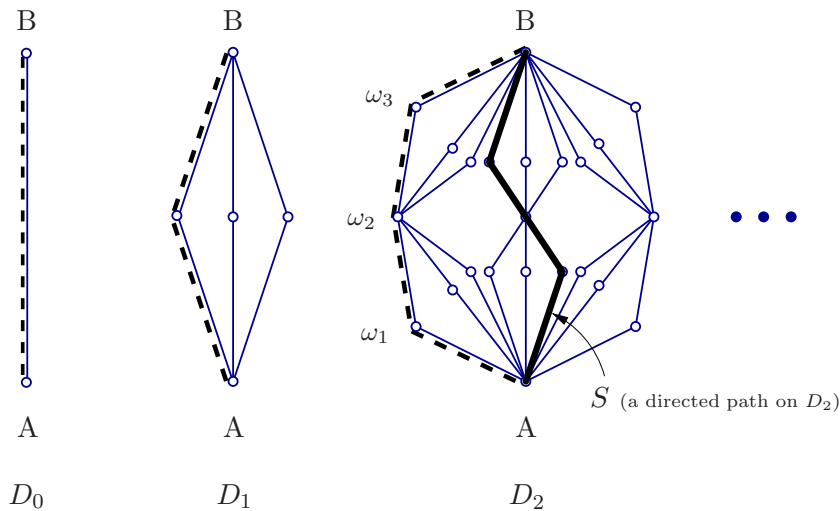}
\end{center}
\caption{\label{fig:Dn} We present here the recursive construction of the first three levels of the hierarchical lattice  $D_n$, for $b=3$, $s=2$. The geometric position of the disordered environment is specified for $D_2$. The law $\bP_n$ is the uniform law over all directed path and   the path $\sigma$
({\sl the wall}) is marked by a dashed line. In the bond disorder case \cite{cf:DHV,cf:GLT} the
hierarchy of lattices is the same, but, with reference to $D_2$, there would be four variables of disorder.}
\end{figure}

\medskip
\begin{rem}
\label{rem:Green}
\rm 
One can directly check that for this model, the site Green function of the directed random walk, that is the probability that a given site is visited by the path, is inhomogeneous (taking value $1$ for the graph extremities $A$ and $B$ and equal to $b^{-i}$ $1\le i\le n$ on the other sites of $D_n$, with $i$ corresponding to the level at which that site has appeared in the hierarchical construction of the lattice, see Figure~\ref{fig:Dn}), and this makes this model similar to the random walk based model
where the Green function decays with a power law with the length of the system
(in the hierarchical context the length of the system is $s^n$). This is not true for the bond model Green function (every bond is visited with probability $b^{-n}$). This inhomogeneity will play a crucial role in the proof, as it will allow us to improve the  method introduced in \cite{cf:GLT}. See Remark~\ref{rem:Green2} for more on this. 
\end{rem}
\medskip

The relatively cumbersome hierarchical construction actually 
boils down to a very simple
 recursion giving the law of the random variable $R_{n+1}$ in terms of the law of $R_n$, for every $n$.
For $s=2$ the recursion is particularly compact:  $R_0:=1$ and  
\begin{equation}
R_n=\frac{R_n^{(1)}A R_n^{(2)}+b-1}{b},
\end{equation}
where $R_n^{(1)}$, $R_n^{(2)}$ and $A$ are
independent random variables with $R_n^{(1)}\stackrel{\cL}=R_n^{(2)}$ and 
\begin{equation}
A \stackrel{\cL}= \exp(\gb\go_1-\log M(\gb)+h).
\end{equation}
For arbitrary $s$ the recursion gets slightly more involved:
\begin{equation}
R_{n+1}=\frac{\prod_{1\le j\le s}R_n^{(j)}\prod_{1\le j\le s-1}A_j+b-1}{b},\label{eq:rec}
\end{equation}
where, once again, all the variables appearing in the right-hand side 
are independent,  $(R_n^{(i)})_{i=1, \ldots, s}$ are also identically
distributed and  $A_j$ has the same law as $A$ for every $j$. 

The expression \eqref{eq:rec} is not only very important on a technical
level, but it
allows an important generalization of the model: there is no reason 
of choosing $b$ integer valued. We will therefore choose it real valued
($b>1$), \rev{while $s$ will always be an integer larger or equal to $2$}.

The quenched disorder $\go$ therefore enters in each step 
of the recursion: the annealed (or pure) model is given by  $r_n=\bbE[R_n]$
and it solves the recursion
\begin{equation}
r_{n+1}=\frac{\exp((s-1)h)r_n^s+(b-1)}{b}, \label{recr}
\end{equation}
with $r_0=1$. It is important to stress that $r_n=R_n^{(1)}$ if $\gb=0$
so, with our choice of the parameters, the annealed case coincides
with the infinite temperature case.
The annealed \textsl{critical point} $h_c$ is the infimum in the set of $h$ such that $r_n$ tends to infinity.
%In analogy with what has been pointed out in \cite[Section 5.3]{cf:GLT},
%it turns out that the cases $b \in (1,s)$ is equivalent to the case
%of $b \in (s, \infty)$ (the case $b=s$
%can be treated apart).  

\rev{The class of systems obtained by choosing $b\in(1,s)$ in \eqref{eq:rec} is of a particular interest. With this setup, the pure system undergoes a phase transition in which the free energy critical exponent can take any value in $(1, \infty)$.
We explain below that the case $b \in [s, \infty)$ can be tackled by our methods too,
but it is less interesting for the viewpoint of the question we are addressing (i.e.\ behavior at marginality).}

%The difficulty in dealing with $b \in (s,\infty)$ comes from the fact that the critical point of the pure system is not explicit (when $b\in (1,s)$ the critical point is just $h=0$). However, the properties of such systems seem to differ significantly from those the former class (the case $b=s$ is a bit different and can be treated apart but does not present a great interest to our purpose).}

%The cases $b \in (s,\infty)$ are very similar to $b\in(1,s)$. In fact for any $b\in(s,\infty)$ one can show (as in \cite[Section 5.3]{cf:GLT}) that there exists a value $\bar{b}\in (1,s)$ such that the critical properties of the two systems are the same.
%The difficulty in dealing with $b \in (s,\infty)$ comes form the fact that the critical point of the pure system is not explicit (when $b\in (1,s)$ the critical point is just $h=0$), and therefore the arguments becomes notationally heavier.
%Therefore for the sake of conciseness we choose to deal only with $b \in (1,s)$

%{\bf Is it true...?}

\subsection{Definition and existence of the free energy}

We are interested in the study of the free energy of this pinning model. The following result ensures that it exists and states some useful technical estimates.

\medskip

\begin{proposition}
 The limit
\begin{equation}
 \lim_{n\to\infty} s^{-n}\log R_n = \tf(\gb,h),
\end{equation}
exists almost surely and it is non-random. Moreover the convergence holds also in $\bbL_1(\dd \bbP)$.
The function $\tf(\gb,\cdot)$ is non-negative, non-decreasing. 
Moreover 
there exists a constant $c$ (depending on $\gb$ and $h$, $b$ and $s$
%, that can be chosen for all $\gb\le 1$ and $h\le 1$ {\bf make it more clear}
) such that
\begin{equation}
 |s^{-n}\bbE \log R_n-\tf(\gb,h)|\le cs^{-n}.
\end{equation}
\end{proposition}

\medskip

This result is the exact equivalent of \cite[Theorem 1.1]{cf:GLT}. Though certain modifications are needed to take into account the difference between the two models, it is straightforward to adapt the method of proof.

We set
\begin{equation}
 h_c(\gb)\, =\inf\{h \text{ such that } \tf(\gb,h)>0)\}\ge 0.
\end{equation}
The value $h_c(\gb)$ is called the \textsl{critical point of the system}, basic properties of the free energy ensure that $\tf(\gb,h)>0$ if and only if $h>h_c(\gb)$. Of course
$h_c(\gb)$ is a non analyticity point of $\tf(\gb, \cdot)$. By {\sl critical behavior of the system} we will refer to how the
free energy vanishes as $h \searrow h_c(\gb)$.

It is not hard to check that with $b\in(1,s)$, we have $h_c(0)=0$.
Indeed if $h<0$, $r_n$ converges to $r_\infty\in(0,1)$ the stable fixed point of the map $x\mapsto (\exp((s-1)h)x^s+(b-1)r_0)/{b}$;
if $h=0$, $r_n=1$ for every $n$;
if $h>0$, $r_n$ diverges to infinity in such a way that  the free energy is  positive,
 therefore $h_c(0)=0$.

Moreover 
there is an immediate comparison between annealed and quenched systems:  by Jensen inequality we have that for any $\gb$ and $h$
\begin{equation}
\bbE \log R_n \, \le\,  \log r_n, 
\end{equation}
so that $\tf(\gb,h)\le \tf(0,h)$ for any $\gb>0$ and $h_c(\gb)\ge h_c(0)=0$.

\subsection{Some results on the free energy}

We state here all the results on properties of the critical system that have been proved for the bond-disorder model in
\cite{cf:GLT} and that are still true in our framework. 
Our first result concerns the shape of the free energy curve around zero in the pure system.

\medskip
\begin{theorem}\label{th:freeen}
For every $b\in (1,s)$ 
there exists a constant $c_b$ such that for all $0\le h\le 1$,
\begin{equation}
 \frac{1}{c_b}h^{1/\alpha}\le \tf(0,h)\le c_b h^{1/\alpha}, \label{eq:anneal}
\end{equation}
where $\alpha:=(\log s-\log b)/{\log s}$.
\end{theorem}

\medskip

The second result describes the influence of the disorder for $b\neq\sqrt{s}$, i.e.\ whether or not the quenched annealed critical point coincide, and it gives an estimate for their difference. Recall that if quenched and annealed critical points differ, we say that the disorder is relevant (and irrelevant if they do not).

\medskip
\begin{theorem}\label{th:relirel}
 When $b\in(\sqrt{s},s)$, for $\gb\le \gb_0$ (depending on $b$ and $s$), we have
 $h_c(\gb)=h_c(0)=0$ and moreover, for any $\gep>0$ the exists $h_\gep$ such that for any $h\le h_\gep$
\begin{equation}
(1-\gep)\tf(0,h)\le \tf(\gb,h)\le \tf(0,h). \label{eq:simsim}
\end{equation}
When $b< \sqrt{s}$ there exists $c_{b,s}$ such that for every $\gb\le1$
\begin{equation}
 \frac{1}{c_{b,s}}\gb^{\frac{2\ga}{2\ga-1}}\le h_c(\gb)-h_c(0)\le c_{b,s}\gb^{\frac{2\ga}{2\ga-1}}.
\end{equation}
\end{theorem}

\rev{As a matter of fact, it can be shown that when $b\ge s$, the annealed free energy grows slower than any power of $(h-h_c(0))$ when $h\to (h_c)_+$. In a sense, $\eqref{eq:anneal}$ holds with $\alpha=0$, and Harris Criterion predicts that $\eqref{eq:simsim}$ holds in that case. One should be able to prove this by using a second moment method approach. For a recent work concerning the case $\alpha=0$ in the non-hierarchical setup, see \cite{cf:AZ_new}}.

\subsection{Main result: the marginal case $b=\sqrt{s}$}

The main novelty of this paper is the result we  present now: 
disorder is relevant for the marginal case $b=\sqrt{s}$.
To our knowledge this is the first example in which one can 
establish the character of the disorder in the marginal case.
 %(in both\cite{cf:GLT} and \cite{cf:KZ, cf:DGLT} this question is left open) 
 \rev{Moreover, we believe that the method of proof
is sufficiently flexible to be adapted to other contexts.}
 
 \medskip
 
\begin{theorem}\label{th:mainres}
When $b=\sqrt{s}$ there exists positive constants $c_1$, $c_2$ and $\gb_0$ (depending on $s$) such that for every $\gb\le \gb_0$ 

%{\bf arent two constants enough?}
% In fact c_3 has to be there because we are not sure wether h_c\le 1 for \gb=1
\begin{equation}\label{eq:tbd}
 \exp\left(-\frac{c_1}{\gb^2}\right) \le h_c(\gb)\le \exp\left(-\frac{c_2}{\gb}\right).
\end{equation}
 \end{theorem}
 
 \medskip

The two bounds are obtained by very different methods, and they will be proven in the two sections that follow. The upper bound for $h_c(\gb)$ is proven by controlling the variance for a finite volume, and using a finite volume estimate for the energy. This is essentially the same method as the one developed in \cite{cf:GLT} and it is analagous to what is done in \cite{cf:Ken} for the model based on a renewal process. The lower bound is obtained using fractional moments, and a change of measure on the environment, in fact a shift of the values taken by $\go$ (in the Gaussian case). The argument  uses strongly the specific property of the site disorder model, in fact the value of the shift is site dependent, in order to take advantage of  the fact that some sites are more likely to be visited than others.
\medskip
\rev{
\begin{rem}\rm \label{rem6}
Contrary to non marginal cases, the two bounds we obtain for $h_c(\gb)$ for $b=\sqrt{s}$ do not match, showing that the understanding of the marginal regime is still incomplete. However:
\begin{itemize}
 \item[(1)] The fact that the lower bound given in Theorem \ref{th:mainres} corresponds to the upper bound found in \cite{cf:GLT} is purely accidental and it should not lead to misleading conclusions.
 \item[(2)] In \cite{cf:DHV}, it is predicted that for the bond model, the second moment method gives the right bound for $h_c(\gb)$. In view of this prediction, the upper bound should give the right order for $h_c(\gb)$, although we do not have any mathematics to support this prediction.
\end{itemize}
\end{rem}}

\begin{rem}\label{rem:Green2}\rm 
We can now make Remark \ref{rem:Green} more precise. The case $b=\sqrt{s}$ on which we focus is really similar the pinning model defined with the simple random walk. 
Indeed, for integer $b$ (recall definition \eqref{eq:part}) The expected number of contacts with the interface $\sigma$ is
\begin{equation}
\bE_n\left(\sum_{i=1}^{s^n-1} \ind_{\{S_i=\sigma_i\}}\right)=\sum_{i=1}^n b^{-i}(s-1)s^{i-1}\,
=\, \frac{s-1}{s-b}(s/b)^n.
\end{equation}
When $b=\sqrt{s}$, it is proportional to $s^{n/2}$ which is the square root of the length of the system. This is also the case for the random walk in dimension $1$ where
$\bE\left(\sum_{i=1}^n \ind_{\{S_i=0\}}\right)$  behaves like $ n^{1/2}$ and we have thus a clear analogy between site hierarchical model and random walk model. On the other hand Remark \ref{rem6}(2) possibly suggests that $h_c(\gb)-h_c(0)$ behaves like $\exp(-c/\gb^2)$ both in the bond hierarchical and random walk model (and not like $\exp(-c/\gb)$). 
\end{rem}
\rev{
\begin{rem}\rm   \label{remrev}
Since the first version of the present paper, some substantial progresses have been made in the understanding of pinning model at marginality.
Using a different method, in \cite{cf:GLT_marg}, Giacomin and Toninelli together with the author proved marginal relevance of disorder for both the hierarchical-lattice with bond disorder model and the random-walk based model. While the method we present here turns out to be less performing than the method in \cite{cf:GLT_marg}
on bond disorder and random walk based models, it should be pointed out
that it does improve on the results in \cite{cf:DGLT} (in the direction of the statements in \cite{cf:KZ}). Moreover
our inhomogeneous shifting procedure, with respect to the more complex change of measure in \cite{cf:GLT_marg},
has the advantage of being very flexible and easier to adapt to more general contexts.
On the other hand, it can be shown that adapting the procedure in \cite{cf:GLT_marg} to the
site disorder model would lead to replacing the exponent $2$ in the left-most side of \eqref{eq:tbd} with $4/3$,
but the proof is substantially heavier than the one that we present, for a result that is still comparable on a qualitative level
(the upper bound is not matched).
\end{rem}}

\medskip

In the sequel we focus on the proof of the case $b=\sqrt{s}$, but the arguments can be adapted (and they get simpler) to prove also the inequalities of 
Theorem \ref{th:relirel}. We stress once again that the case $b\neq\sqrt{s}$ 
is  detailed in \cite{cf:GLT} 
for the bond model.

\section{The upper bound: control of the variance}
The main result of this section is:

\begin{proposition}\label{th:lvbds}
For any fixed $s$, and $b=\sqrt{s}$ one can find constants $c_s$ and $\gb_0$, such that for all $\gb\le \gb_0$,
\begin{equation}
h_c(\gb)\le \exp\left(-\frac{c_s}{\gb}\right).
\end{equation}
\end{proposition}
Such a result is achieved by a second moment computation and for this we have to get some bounds on the first two moments of $R_n$.

\subsection{A lower bound on the growth of $r_n$}
We prove a technical result on the growth of $r_n$ when $h>0$.
For convenience we write $p_n:=(r_n-1)$. We have $p_0=0$ and \eqref{recr} becomes
\begin{equation}
p_{n+1}=\frac{1}{b}\left((1+p_n)^s\exp((s-1)h)-1\right). \label{recrr}
\end{equation}
%\begin{lemma}\label{th:pn}
%For any $h>0$ and any $n\ge 1$ we have $p_n\ge \frac{h}{b}(s/b)^{n-1}$.
%\end{lemma}
%\begin{proof}
Note that if $h>0$, $p_n>0$ for all $n \ge 1$, so that \eqref{recrr} implies
\begin{equation}
p_{n+1}\ge \frac{s}{b} p_n+\frac{h(s-1)}{b}\ge \frac{s}{b}p_n+\frac{h}{b}, \label{rect}
\end{equation}
and therefore 
\begin{equation}\label{pnge}
p_n\ge (s/b)^{n-1} (h/b).
\end{equation} 

%One can find a constant $c_1$ such that, as long as $p_n\le 1$, and for all $h\le 1$ we have
%\begin{equation}
%p_{n+1}\le \left(\frac{s}{b}p_n+ c_ 1h\right) (1+c_1 p_n). \label{eq:pn}
%\end{equation}
 
%By induction we have that if $p_{n-1}\le 1$

%\begin{align*}
%p_n\le \left(c_1 h\sum_{i=0}^{n-1}\left(\frac{s}{b}\right)^i\right)\prod_{i=0}^{n-1}(1+c_1 p_i)
%\end{align*}

%Moreover, as $p_i\ge \frac{s}{b} p_{i-1}$ we have,
%\begin{align*}
%\prod_{i=0}^{n-1}(1+cp_i)\le \prod_{i=0}^{\infty}\left(1+\left(\frac{s}{b}\right)^{-n}\right)\le c_2
%\end{align*}

%Where the constant $c_2$ depends just on $s$ and $b$.
%Therefore, as long as $p_{n-1}\le 1$

%\begin{equation}
% p_n\le c_3 \left(\frac{s}{b}\right)^{n} h. \label{eq:upbd}
%\end{equation}
%By induction, one can see that this is true in particular for $n \le \frac{|\log C_2 p_0|}{\log s-\log b}$ when $C_2p_0\le 1$. 

%Equation \eqref{eq:pn} also garanties that

%\begin{equation}
% p_n\ge \left(\frac{s}{b}\right)^{n} p_0. \label{eq:lwbd}
%\end{equation}

%For any $n\ge 0$.

\subsection{An upper bound for the growth the variance}
We prove now a technical result concerning the variance of $R_n$ which is crucial for the proof of Proposition \ref{th:lvbds}.
Before stating the result we introduce some notation and write the induction equation for the variance.
The variance $\Delta_n$ of the random variable $R_n$ is given by the following recursion

\begin{equation}
 \gD_{n+1}=\frac{1}{b^2}\left(\left(\Delta_n+r_n^{2}\right)^s\exp\left((s-1)(\gga(\gb)+2h)\right)-r_n^{2s}\exp(2(s-1)h)\right),
\end{equation}
where $\gga(\gb)=\log M(2\gb)-2\log M(\gb)$ (recall that $M(\gb)=\bbE[\exp(\gb\go_1)])$. Because $\go$ has unit variance, we have $\gga(\gb)\stackrel{\gb\searrow 0}{\sim} \gb^2$.\\
Let $v_n$ denote the relative variance $\gD_n/(r_n)^2$.
We have
\begin{equation} \label{eq:relv}
v_{n+1}=\frac{b^2r_n^{2s}\exp\left((s-1)2h\right)}{\left(r_n^s\exp\left((s-1)h\right)+(b-1)\right)^2}\frac{\exp((s-1)\gga(\gb))(v_n+1)^s-1}{b^2}.
\end{equation}
Let $n_1$ be the smallest integer such that $p_n\ge 1$. 

\begin{lemma}\label{th:dddd}
We can find constants $c_5$ and $\gb_0$ such that for all $\gb<\gb_0$,
for $h=\exp(-c_5/\gb)$,
\begin{equation}
v_{n_1}\le \gb
\end{equation}
\end{lemma}

%\begin{lemma}
%When $b=\sqrt{s}$, there exists some constant $c_5$ such that 
%for any $\gb\le \gb_0$ small enough and $h=\exp(\frac{c_5}{\gb})$,
%for 
%$v_n\le \gep$.
%\end{lemma}

\begin{proof}
We make a Taylor expansion of \eqref{eq:relv} around $r_n=1$, $h=0$, $\gb=0$, $v_n=0$,

\begin{equation}
v_{n+1}=\left(1+O(h+p_n)\right)\left(\frac{(s-1)}{b^2}\gb^2+\frac{s}{b^2}v_n+O(v_n^2)+o(\gb^2)\right).
\end{equation}
From the previous line, we can find a constant $c_4$ such that if $0<p_n\le 1$, $h\le1$, $\gb\le1$ and $v_n\le 1$ we have (recall that $b=\sqrt{s}$)
\begin{equation}
v_{n+1}\le c_4\gb^2+v_n(1+c_4v_n)(1+c_4(h+p_n)).
\end{equation}
By induction, we get that as long if $v_{n-1}\le 1$ and $p_n\le 1$, we have
\begin{equation}
v_n\le n c_4\gb^2\prod_{i=0}^{n-1}(1+c_4v_i)(1+c_4(h+p_i)). 
\end{equation}
\rev{By \eqref{pnge} we have $(h+p_i)\le (1+b)p_i$ for all $i\ge 1$}. Changing the constant $c_4$ if necessary we get the nicer formula
\begin{equation}
v_n\le n c_4\gb^2\prod_{i=1}^{n-1}(1+c_4v_i)(1+c_4p_i). \label{nnn}
\end{equation}
Let $n_0$ be the smallest integer such that $v_{n_0}\ge \gb$. We have to show that we cannot have $n_0\le n_1$.
If $n_0\le n_1$, \eqref{nnn} implies
\begin{equation}
v_{n_0}\le n_0 c_4 \gb^2 \prod_{i=1}^{n_0-1}(1+c_4v_i)(1+c_4 p_i). \label{eq:pff}
\end{equation}
As $p_{n+1}\ge (s/b)\, p_n$ for all $n\ge 0$ (cf.\ \eqref{rect}), and $p_{n_1-1}\le 1$, we have $p_{n_1-2}\le (b/s)$,\\
 $p_{n_1-3}\le (s/b)^{-2}$ and by induction for all $i\le n_1 -1$
\begin{equation}
 p_i\le (s/b)^{i-(n_1-1)}.
\end{equation}
Therefore
\begin{equation}
\prod_{i=1}^{n_0-1} \left(1+c_4 p_i\right)\le\prod_{i=1}^{n_1-1} \left(1+c_4 p_i\right)\le \prod_{i=1}^{n_1-1} \left(1+c_4 (s/b)^{i-(n_1-1)}\right)\le \prod_{k=0}^{\infty}\left(1+c_4(s/b)^{-k}\right).
\end{equation}
The last term is finite, and is clearly not dependent on $\gb$ or $h$.
Moreover, because $p_n\ge (s/b)^n (h/b)$, it is necessary that 
\begin{equation}
n_1-2\le \frac{\log (b/h)}{\log (s/b)}.
\end{equation}
Replacing $b$ and $h$ with $\sqrt{s}$ and $\exp\left(-\frac{c_5}{\gb}\right)$, we get
\begin{equation}
n_0\le n_1\le \frac{2c_5}{\gb\log s}+1\le \frac{3c_5}{\gb\log s}
\end{equation}
for $\gb$ small enough.
Replacing $n_0$ by this upper bound in \eqref{eq:pff} gives us that $n_0\le n_1$ implies
\begin{equation}
\gb\le v_{n_0}\le \gb \left[\frac{3c_5c_4}{\log s}(1+c_4\gb)^{\frac{3c_5}{\gb\log s}}\prod_{k=0}^{\infty}(1+c_4 (s/b)^{-k})\right].
\end{equation}
If $c_5$ is chosen small enough the right-hand side is smaller than the left--hand side.
\end{proof}

\subsection{Proof of Proposition \ref{th:lvbds}}

Let us choose $c_5$ as in Lemma \ref{th:dddd}, $\gb$ small enough, and $h=\exp(-c_5/\gb)$. We fix some small $\gep>0$.
From Lemma \ref{th:dddd}, we have $v_{n_1}\le \gb$ and $r_{n_1}\ge 2$. The idea of the proof is to consider some $n$ a bit larger that $n_1$ such that $r_n$ is big and $v_n$ is small, in order to get a good bound on $\bbE \log R_n$.

We use \eqref{eq:relv} to get a rough bound on the growth of $v_n$ when $n\ge n_1$,
\begin{equation}
v_{n+1}\le (1+v_n)^s\exp\left[\gga(\gb)(s-1)\right]-1.
\end{equation}
Hence, one can find a constant $c_6$ such that as long as $v_n\le 1$ and $\gb$ small enough, we have
\begin{equation}
v_{n+1}\le c_6(v_n+\gb^2).
\end{equation}
If we choose $c_6>1$, this implies that for any integer $k\ge 0$
\begin{equation}
v_{n_1+k}\le c_6^k(\gb+k\gb^2),
\end{equation}
provided the right--hand side is less than $1$.
\medskip

We fix $k$ large enough, and $\gb_0$ such that  $c_6^k(\gb+k\gb^2)\le \gep$ for all $\gb\le \gb_0$.
From \eqref{rect} and the definition of $n_1$, we have $r_{n_1+k}\ge 1+(s/b)^k$. Let $k$ be a fixed (large) integer, \rev{Chebycheff inequality implies that}
\begin{equation}
\bbP\left(R_{n_1+k}\le (1/2) r_{n_1+k}\right)\le 4v_{n_1+k}\le 4\gep.
\end{equation}
We write $n_2=n_1+k$. Using the fact that $R_n\ge (b-1)/b$ we have
\begin{multline}
 \bbE\left[\log R_{n_2}\right]\ge  \left[\log r_{n_2}-\log 2\right]\bbP\left(R_{n_2}\ge (1/2) r_{n_2}\right)+\log\frac{b-1}{b}\bbP\left(R_{n_2}\le (1/2)r_{n_2}\right)\\
										\ge  (1-4\gep)\left[\log (1+(s/b)^k)-\log 2\right]-4\gep \log \frac{b}{b-1}.
\end{multline}
By choosing a suitable $k$, this can be made arbitrarily large.
Taking the $\log$ in \eqref{eq:rec} and forgetting the $(b-1)$ term gives
\begin{equation}
\bbE \log R_{n+1}\ge s \bbE\left[\log R_{n}\right]+(s-1)\left(h-\log M(\gb)+\bbE[\go_1]\right)-\log b.
\end{equation} 
Therefore, the sequence $s^{-n}\left[\bbE\left[\log R_n\right] -\frac{\log{b}}{s-1}+h-\log M(\gb)\right]$ is increasing (recall $\bbE[\go_1]=0$).
With our settings we have
\begin{equation}
\bbE\left[\log R_{n_2}\right]> \frac{\log{b}}{s-1}+\log M(\gb)-h,
\end{equation}
therefore 
\begin{equation}
\tf(\gb,h)=\lim_{n\rightarrow\infty}s^{-n}\left[\bbE\left[\log R_n\right] -\frac{\log{b}}{s-1}+h-\log M(\gb)\right]>0.
\end{equation}
\qed

\section{The lower bound: Fractional moment and improved shifting method}

In this section, we prove the lower bound by improving the method of measure-shifting used in \cite{cf:GLT} and \cite{cf:DGLT}. Instead of considering an homogeneous shift on the environment, we chose to shift more the sites that are more likely to be visited.

\begin{proposition}\label{th:upbd}
When $b=\sqrt{s}$, there exists a constant $c_s$ such that for all $\gb\le 1$ we have
\begin{equation}
h_c\ge \exp(-c_s/\gb^2).
\end{equation}
\end{proposition}
\subsection{Fractional moment}\label{th:fracmom}

\begin{lemma}
Fix $\theta\in(0,1)$ and set
\begin{equation}
x_{\theta}=\max\left\{ x \ \Big| \ \frac{x^s+(b-1)^{\theta}}{b^{\theta}}\le x\right\}.
\end{equation}
\rev{(Note that $x_{\theta}$ is defined whenever $\theta$ is close enough to $1$. When it is defined we have $x_{\theta}<1$ as the inequality cannot be fulfilled for $x\ge 1$.)}\\
 If $\bbE\left[\exp\left(\theta(\gb\go_1-\log M(\gb)+h)\right)\right]\le 1$, and if there exists $n$ such that 
 $\bbE[R_n^{\theta}]\le x_\theta$, then $\tf(\gb,h)=0$.
\end{lemma}

\begin{proof}

Let $0<\theta<1$ be fixed, an $u_n=\bbE[ R_n^{\theta}]$ denotes the fractional moment of $R_n$.
We write $a_{\theta}=\bbE[ A_1^{\theta}]=\bbE\left[\exp\left(\theta(\gb\go_1-\log M(\gb)+h)\right)\right]$.
Using the basic inequality $\left(\sum x_i\right)^{\theta}\le \sum x_i^{\theta}$ and averaging with respect to $\bbP$,
we get from \eqref{eq:rec}

\begin{equation}
 u_{i+1}\le \frac{u_i^s a_{\theta}^{s-1}+(b-1)^{\theta}}{b^{\theta}}.
\end{equation}
With our assumption, $a_{\theta}\le 1$, so that
\begin{equation}
 u_{i+1}\le \frac{u_i^s+(b-1)^{\theta}}{b^{\theta}}.
\end{equation}
The map 
\begin{equation}
g_\theta:\ x\mapsto \frac{x^s+(b-1)^{\theta}}{b^{\theta}}\quad \text{ for } x\ge 0 \label{eq:gg},
\end{equation}
 is non-decreasing so that if $u_n\le x_\theta$ for some $n$, then $u_i\le x_\theta$ for every $i\ge n$.
In this case the free energy is $0$ as
\begin{equation}
\tf(\gb,h)=\lim_{n\to\infty}s^{-n}\bbE \log R_n\le \liminf_{n\to\infty}\frac{1}{\theta s^n}\log u_n=0,
\end{equation}
where the last inequality is just Jensen inequality.
\end{proof}

We add a second result that guaranties that the previous lemma is useful.

\begin{lemma}
\begin{equation}
\lim_{\theta\rightarrow 1^{-}} x_{\theta}=1.
\end{equation}
\end{lemma}

\begin{proof}
One just has to check that 
\begin{equation}
\lim_{\theta\rightarrow 1^{-}}g_{\theta}(1-\gep)=\frac{(1-\gep)^s+(b-1)}{b}<1-\gep,
\end{equation}
if $\gep$ is small enough.
\end{proof}

\subsection{Improved shifting method}

In this section, we prove Proposition \ref{th:upbd}, by estimating the fractional moment of $R_n$ by making a measure change on the environnement, making $\go$ lower. We simply use H\"older inequality to estimate the cost of the change of measure.
For simplicity we first write the proof for gaussian environment.

\begin{proof}[Proof of Proposition \ref{th:upbd}, Gaussian case]

Let $\gep>0$ be small (we will fix conditions on it later). We chose $\theta<1$ (close to 1) such that $x_{\theta}\ge 1-\gep$,
and some small $\eta>0$ whose value will depend on $\gep$. We consider (for $\gb\le 1$), the system of size $n=\frac{1}{\eta^2\gb^2}$ (without loss of generality, we can suppose it to be an integer) and $h=s^{-n}=\exp\left(-\frac{\ln s}{\eta^2\gb^2}\right)$. One can check that the condition $\bbE\left[\exp\left(\theta(\gb\go_1-\log M(\gb)+h)\right)\right]\le 1$ is fulfilled for $\gb$ small enough.

We define sets $V_i$ for $0\le i< n$ by
\begin{equation}
V_i=\left\{ j\in\{1,\dots,s^n-1\} \text{ such that } s^i \text{ divides } j \text{ and } s^{j+1} \text{does not divide } j \right\}  \label{eq:vi}
\end{equation}
Note that $|V_i|=(s-1)s^{n-1-i}$.

We slightly modify the measure $\bbP$ of the environment by shifting the value of $\go_j$ for $j\in V_i$, $i< n$ by $\frac{\eta s^{i/2}}{s^{n/2}\sqrt{n}}=\gd_i$ and we call $\tilde \bbP$ the modified measure. Notice that $h$ is small compared to any of the $\gd_i$.
The density of this measure is

\begin{equation}
 \frac{\dd \tilde\bbP}{\dd \bbP}(\go)=\exp\left(-\sum_{0\le i \le n-1}\sum_{j\in V_i}\left(\gd_i\go_j+\frac{\gd_i^2}{2}\right)\right).
\end{equation}
In order to estimate $u_n$ we use H\"older inequality
\begin{equation}
\bbE [ R_n^{\theta}] = \tilde\bbE\left[\frac{\dd\bbP}{\dd \tilde \bbP}R_n^{\theta} \right]\le\left( \tilde\bbE \left[\left(\frac{\dd\bbP}{\dd \tilde \bbP}\right)^{1/(1-\theta)}\right]\right)^{1-\theta} \left(\tilde \bbE\left[ R_n \right]\right)^{\theta}. \label{eq:hld}
\end{equation}
The first term is computed explicitly with the expression of the density and it is equal to
\begin{multline}
  \left( \tilde\bbE \left[\left(\frac{\dd\bbP}{\dd\tilde \bbP}\right)^ 
{1/(1-\theta)}\right]\right)^{1-\theta}\, =\,
\exp\left(\frac{\theta}{2 (1-\theta)}\sum_{i=0}^{n-1} |V_i|\gd_i^2 
\right)\, \\=
\exp\left(\frac{\theta}{2(1-\theta)}\sum_{i=0}^{n-1}(s-1)s^{n-i-1} 
\frac{\eta^2 s^i}{s^n n}\right)\, =\, \exp\left(\frac{\eta^2\theta  
(s-1)}{2(1-\theta)s}\right), \label{eq:deg}
  \end{multline}
  and we can choose $\eta$ to be such that the right-hand side is less than $1+\gep/2$.
 
Therefore in order to get an upper bound for $u_n$ we have to estimate $\tilde\bbE [R_n]$.
To do this, we write down the recursion giving $\tilde{r}_i=\tilde\bbE [R_i^{(1)}]$, for $i\le n-1$.
We have  $\tilde r_0=1$ and
\begin{equation}
 \tilde r_{i+1}\, =\, \frac{\tilde r_{i}^s\exp\left[(s-1)(-\gb\gd_{i}+h)\right]+(b-1)}{b}\label{eq:shiftrec}.
\end{equation}

%With the choice we have made for $h$, we can notice that $r_n$ is decreasing.
Let us look at the evolution of $g_i=1-\tilde r_i\ge 0$ for $i\le n$.
We have
\begin{equation}\label{okok}\begin{split}
g_{i+1}&=\frac{1}{b}\left[1-(1-g_i)^{s}\exp\left((s-1)(-\gb\gd_{i}+h)\right)\right]\\
       &=\frac{1}{b}\left[1-(1-g_i)^{s}+(1-g_i)^{s}\left(1-\exp\left((s-1)(-\gb\gd_{i}+h)\right)\right)\right].
\end{split}\end{equation}
We can find $c_1$ such that $1-(1-g)^s\ge s(g- c_1 g^2)$ for all $0\le g\le 1$.
By choosing $\beta$ sufficiently small, the term in the exponential is small enough and $h$ is negligible compared to $\gb\gd_i$ so that when $g_i\le 2\gep$
\begin{equation} 
(1-g_i)^{s}\left(1-\exp\left((s-1)(-\gb\gd_{i}+h)\right)\right)\, \ge\, \frac{\gb\gd_i}{2}.
\end{equation}
Therefore, as long as $g_i\le 2\gep$, we have
\begin{equation} \label{eq:pii}
g_{i+1}\ge \frac{s}{b}(g_i-c_1g_i^2)+\frac{\gb\gd_i}{2b}.
\end{equation}

\begin{lemma}
If $\gep$ is chosen small enough (not depending on $\gb$), $g_{n}\ge 2\gep$.
\end{lemma}
\begin{proof}
\rev{Suppose that $g_n\le 2\gep$. Equation \eqref{okok} implies that
\begin{equation}
 g_i\ge \frac{1}{b}[1-(1-g_{i-1})^s], \quad \forall i\le n.
\end{equation}
This is equivalent to
\begin{equation}
 g_{i-1}\le 1-(1-b g_i)^{1/s},
\end{equation}
so that if $g_i\le 2\gep$ and $\gep$ is small enough,
\begin{equation}
 g_{i-1}\le s^{-1/4}g_{i}.
\end{equation}
Using the above equation inductively from $i=n$ to $k+1$, one gets 

%\begin{equation}
%g_i\ge  \frac{s}{b}\left(1- 2 \gep c_1\right)g_{i-1}\ge s^{1/4}g_{i-1},  \quad \forall i\le n+1,
% \end{equation}
%where the last inequality hold if $\gep$ is chosen small enough (recall that $b=\sqrt{s}$).
%Hence we have 
\begin{equation}\label{s1/4}
g_k\le 2\gep s^{(k-n)/4}, \quad \forall k\le n.
\end{equation}
%(in particular $g_{k}\le 2\gep$ for all $k$ smaller than $n$).}

\rev{We write $q_i=s^{(n-i)/2}g_i=\frac{s^{n-i}}{b^{n-i}}g_i$, note that $q_0=0$.
When $g_i\le 2\gep$, from \eqref{eq:pii} and the definition of $q_i$, we have
\begin{equation}
q_{i+1}\ge q_i(1-c_1 g_i)+\frac{\gb\eta}{2b\sqrt{n}}.
\end{equation}
Using the fact that $g_k\le 2\gep$ for all $k\le n$, and the above inequality for $1\le k\le n-1$ we get that
\begin{equation}
 q_n\ge \frac{n\gb\eta}{2b\sqrt{n}}\prod_{k=0}^{n-1} (1-c_1 g_k).
\end{equation}
Now to we use \eqref{s1/4} to get that
  \begin{equation}
  \prod_{k=0}^{n-1} \left(1-c_1 g_k\right)\ge \prod_{k=0}^{n-1} \left(1-c_1 2\gep s^{(k-n)/4}\right)\ge \prod_{i=1}^{\infty}\left(1-c_1 2\gep s^{-i/4}\right) \ge 1/2,
  \end{equation}
where the last inequality holds if $\gep$ is chosen small enough.}
Hence
 
\begin{equation}
 g_n=q_n \ge \frac{n\gb\eta}{2b\sqrt{n}}\prod_{k=0}^{n-1} (1-c_1 g_k)\ge \frac{\gb\eta\sqrt{n}}{4b}=\frac{1}{4b},
\end{equation}
and from that we infer that $g_n\ge 2\gep$ if $\gep<1/(8b)$.}
\end{proof}

We just proved that $\tilde r_n \le 1-2\gep$. Using \eqref{eq:hld}, and the bound we have on \eqref{eq:deg} we get
\begin{equation}
u_n\le (1-2\gep)^{\theta}(1+\gep/2)\le 1-\gep\le x_\theta.
\end{equation}
The last inequality is just comes from our choice for $\theta$, the second inequality is true if $\gep$ is small, and $\theta$ sufficiently close to one. The result follows from Lemma \ref{th:fracmom}
\end{proof}

\begin{proof}[Proof of Proposition \ref{th:upbd}, general non-gaussian case]

When the environment is\\ non--gaussian,
one can generalize the preceding proof by making a tilt on the measure instead of a shift. This means that the change of measure becomes
\begin{equation}
\frac{\dd \tilde \bbP}{\dd \bbP}(\go)=\exp\left(-\sum_{0\le i \le n-1}\sum_{j\in V_i}\left(\gd_i\go_j+\log M(-\gd_i)\right)\right).
\end{equation}

Therefore the term giving the cost of the change of measure (cf. \eqref{eq:hld}) is

\begin{equation} \begin{split}
 \left( \tilde\bbE \left[\left(\frac{\dd\bbP}{\dd\tilde \bbP}\right)^{1/(1-\theta)}\right]\right)^{1-\theta}&=
\exp\left((1-\theta)\sum_{i=0}^{n-1} |V_i|\left[\log M\left(\frac{\theta}{1-\theta}\gd_i\right)+\frac{\theta}{1-\theta}\log M(-\gd_i)\right]\right)\\
&\le \exp\left(\frac{\theta}{(1-\theta)}\sum_{i=0}^{n-1}(s-1)s^{n-i-1}\frac{\eta^2 s^i}{s^n n}\right)=\exp\left(\frac{\eta^2\theta (s-1)}{(1-\theta)s}\right).
 \end{split}\end{equation}

Where the inequality is obtained by using that $\log M(\gb)\sim \gb^2/2$, so that for $x$ small enough,
$\log M (x)\le x^2$.

The second point where we have to look is the estimate of $\tilde r_n$. In fact, the term $\exp((s-1)(-\gb\gd_i+h))$ should be replaced by
\begin{equation}
\exp\left[(s-1)\left(\log M(\gb-\gd_i)-\log M(\gb)-\log M(-\gd_i)+h\right)\right].
\end{equation}
Knowing the behavior of the convex function $\log M(\cdot)$ near zero, it is not difficult to see that this is less than $\exp(-c_6\gb\gd_i+(s-1)h)$ for some constant $c_6$, which is enough for the rest of the computations.
\end{proof}

\section*{Acknowledgments}
The author is very grateful to Giambattista Giacomin for suggesting to work on this subject, for enlightening discussion and many useful comments concerning the manuscript. The comments and the careful reviewing of an anonymous referee have been of great help too.
The author also acknowledges the support of ANR, grant POLINTBIO.

\end{document}